\documentclass[11pt]{article}
\usepackage{amsmath}
\usepackage{amsfonts}
\usepackage{amssymb}
\usepackage{mathrsfs}
\usepackage{amsthm}
\usepackage{palatino}
\usepackage[margin=2.5cm, vmargin={1.5cm}]{geometry}

\usepackage{dcolumn}

\usepackage{times}
\usepackage{graphicx}
\usepackage{xcolor}

\usepackage{pstricks}
\usepackage{pst-plot}

\begin{document}

\title{De Moivre and Bell polynomials}

\author{
  Cormac ~O'Sullivan\footnote{{\it Date:} Mar 5, 2022.
\newline \indent \ \ \
  {\it 2010 Mathematics Subject Classification:} 05A15, 05A16, 13F25.
  \newline \indent \ \ \
Support for this project was provided by a PSC-CUNY Award, jointly funded by The Professional Staff Congress and The City
\newline \indent \ \ \
University of New York.}
  }

\date{}

\maketitle

\def\s#1#2{\langle \,#1 , #2 \,\rangle}

\def\F{{\frak F}}
\def\C{{\mathbb C}}
\def\R{{\mathbb R}}
\def\Z{{\mathbb Z}}
\def\Q{{\mathbb Q}}
\def\N{{\mathbb N}}
\def\G{{\Gamma}}
\def\GH{{\G \backslash \H}}
\def\g{{\gamma}}
\def\L{{\Lambda}}
\def\ee{{\varepsilon}}
\def\K{{\mathcal K}}
\def\Re{\mathrm{Re}}
\def\Im{\mathrm{Im}}
\def\PSL{\mathrm{PSL}}
\def\SL{\mathrm{SL}}
\def\Vol{\operatorname{Vol}}
\def\lqs{\leqslant}
\def\gqs{\geqslant}
\def\sgn{\operatorname{sgn}}
\def\res{\operatornamewithlimits{Res}}
\def\li{\operatorname{Li_2}}
\def\lip{\operatorname{Li}'_2}
\def\pl{\operatorname{Li}}

\def\dm{{\mathcal A}}
\def\bl{{\mathcal B}}
\def\by{{\mathcal Y}}
\def\hr{{C}}

\def\clp{\operatorname{Cl}'_2}
\def\clpp{\operatorname{Cl}''_2}
\def\farey{\mathscr F}

\newcommand{\stira}[2]{{\genfrac{[}{]}{0pt}{}{#1}{#2}}}
\newcommand{\stirb}[2]{{\genfrac{\{}{\}}{0pt}{}{#1}{#2}}}
\newcommand{\norm}[1]{\left\lVert #1 \right\rVert}

\newcommand{\e}{\eqref}
\newcommand{\bo}[1]{O\left( #1 \right)}

\newtheorem{theorem}{Theorem}[section]
\newtheorem{lemma}[theorem]{Lemma}
\newtheorem{prop}[theorem]{Proposition}
\newtheorem{conj}[theorem]{Conjecture}
\newtheorem{cor}[theorem]{Corollary}
\newtheorem{assume}[theorem]{Assumptions}
\newtheorem{adef}[theorem]{Definition}
\newtheorem{eg}[theorem]{Example}

\numberwithin{equation}{section}

\let\originalleft\left
\let\originalright\right
\renewcommand{\left}{\mathopen{}\mathclose\bgroup\originalleft}
\renewcommand{\right}{\aftergroup\egroup\originalright}

\bibliographystyle{alpha}

\begin{abstract}
We survey a family of polynomials that are very useful in all kinds  of power series manipulations, and appearing more frequently in the literature. Applications to formal power series, generating functions and asymptotic expansions are described, and we discuss the related work of De Moivre, Arbogast and Bell.
\end{abstract}

\section{Introduction}
Let $a_1 x +a_2 x^2+ a_3 x^3+ \cdots $ be a formal power series without a constant term and with coefficients in $\C$. It could represent a function in some neighborhood of $x=0$, but usually we don't require convergence.  The $k$th power of this series may be expanded into another series
\begin{equation} \label{bell}
    \left( a_1 x +a_2 x^2+ a_3 x^3+ \cdots \right)^k = \sum_{n\in \Z} \dm_{n,k}(a_1, a_2, a_3, \dots) x^n \qquad \quad (k \in \Z_{\gqs 0}),
\end{equation}
and it can be seen that
 the new coefficients $\dm_{n,k}(a_1, a_2, a_3, \dots)$ can only be nonzero if $n\gqs k$, in which case they  depend at most on $a_1, a_2, \dots, a_{m}$ for $m=n-k+1$. The multinomial development
\begin{equation*}
  \left( a_1 x +a_2 x^2+  \cdots +a_m x^m \right)^k
  = \sum_{j_1+ j_2+ \dots + j_m= k}
   \binom{k}{j_1 , j_2 ,  \dots , j_m} (a_1x)^{j_1} (a_2 x^2)^{j_2}  \cdots (a_m x^m)^{j_m}
\end{equation*}
shows that
\begin{equation} \label{bell2}
  \dm_{n,k}(a_1, a_2, a_3, \dots) = \sum_{\substack{1j_1+2 j_2+ \dots +mj_m= n \\ j_1+ j_2+ \dots +j_m= k}}
 \binom{k}{j_1 , j_2 ,  \dots , j_m} a_1^{j_1} a_2^{j_2}  \cdots a_m^{j_m}
\end{equation}
where the sum   is over all possible $j_1$, $j_2$,  \dots , $j_m \in \Z_{\gqs 0}$ and $0^0=1$ throughout. 

Thus $\dm_{n,k}$ is a simply described polynomial in $a_1,a_2, \dots, a_{n-k+1}$.   In the literature they are designated  `partial ordinary Bell polynomials', and related to the better-known partial Bell polynomials $\bl_{n,k}$. We will make the case  that in fact the $\dm_{n,k}$ polynomials are the more natural and useful version. Prior to Bell, Arbogast \cite{Arb} was already working with them in 1800. However, their earliest appearance seems to be in a 1697 paper  of De Moivre \cite{dem}. There he says he was curious to see if he could generalize to \e{bell} his friend Mr. Newton's work with binomials. His solution is equivalent to \e{bell2}, though uses an interesting recurrence to describe which products $a_1^{j_1} a_2^{j_2}  \cdots a_m^{j_m}$ appear. This is discussed in section \ref{gfn}. See also \cite[Sect. 5.1]{ScIv} for historical context. We propose a new, more succinct name for these fundamental objects:

\begin{adef} \label{dbf}
{\rm For  $n$,  $k \in \Z$ with $k\gqs 0$, define the {\em De Moivre polynomial} $\dm_{n,k}(a_1, a_2, \dots)$ by \e{bell}. If $n<k$ it is $0$. If $n\gqs k$ it is given by  \e{bell2}, making  a polynomial in $a_1, a_2, \dots, a_{n-k+1}$ of homogeneous degree $k$ with positive integer coefficients and  number of terms equalling the number of partitions of $n$ with $k$ parts; see \e{pnk}.}
\end{adef}



In this article we aim to provide a convenient reference for these polynomials, while also  describing some new results about them. 
The author has found them to be of great use in giving explicit forms for asymptotic expansions, and we will see examples of this in section \ref{laps}. More generally, as shown in section \ref{manx},  composing,  inverting and taking powers of generating functions and power series becomes easy with the help of the De Moivre polynomials, giving clear expressions for the new coefficients. We find in section \ref{gfn} that some familiar generating functions take on a new look with this treatment.  Many applications are discussed; we may also mention \cite[Chap. 1]{Pit}, \cite[Chap. 11]{chacha}, for example, for further uses of De Moivre{\textbackslash}Bell polynomials in  probability and statistics and for  relations among symmetric polynomials.


Before reviewing  the basic properties of $\dm_{n,k}$ in the next section, we complete this introduction by including more of the history of these ideas. Replacing the ordinary series in \e{bell} with an exponential series, (in other words a Taylor series), defines the partial Bell polynomials:
\begin{equation} \label{bellb}
    \frac{1}{k!}\left( a_1\frac{x}{1!}  +a_2\frac{x^2}{2!}  + a_3\frac{x^3}{3!}  + \cdots \right)^k = \sum_{n=0}^{\infty} \bl_{n,k}(a_1, a_2, a_3, \dots)\frac{x^n}{n!} \qquad \quad (k \in \Z_{\gqs 0}).
\end{equation}
(We are following Comtet's terminology from \cite[Chap. 3]{Comtet}.)
Hence we have the relation
\begin{equation} \label{ab}
  \bl_{n,k}(a_1,a_2,a_3, \dots)  = \frac{n!}{k!}\dm_{n,k}(\frac{a_1}{1!},\frac{a_2}{2!},\frac{a_3}{3!}, \dots).
\end{equation}
The complete Bell polynomials are defined as
\begin{equation*}
  \by_{n}(a_1,a_2,a_3, \dots)  := \sum_{k=0}^n \bl_{n,k}(a_1,a_2,a_3, \dots).
\end{equation*}
There is also a related Bell polynomial in one variable:
\begin{equation*}
  \bl_n(x):= n! [t^n] e^{x(e^t-1)} = \sum_{k=0}^n \bl_{n,k}(1,1,1, \dots) x^k
\end{equation*}
where $[t^n]$ indicates the coefficient of $t^n$ in a series.

Bell  introduced the complete polynomials $\by_{n}$ in \cite{Bell34} as a wide generalization of the Hermite and Appell polynomials. They are also related to the partition polynomials he was previously considering - see \cite[Thm. 6.6]{OSsym}. Riordan studied $\by_{n}$ and $\bl_{n,k}$ in \cite{Ri68}, naming them Bell polynomials. Comtet emphasized that what he termed  the `partial ordinary' version should be used  when dealing with ordinary series  like $a_0+a_1 x+a_2 x^2+\cdots$, (he used the notation $\hat{\bl}_{n,k}$ but we prefer  $\dm_{n,k}$ to clearly differentiate the versions).
However, all the formulas in \cite[Chap. 3]{Comtet} involve  the  polynomials $\bl_{n,k}$ and this might account for their dominance in the literature.

 Research into the origins of Fa\`a di Bruno's formula in \cite[pp. 52, 481 - 483]{Knuprog} and \cite{Cr05,Jo02} uncovered the work of Arbogast  in \cite{Arb} where the formula in fact first appeared.  Arbogast  takes powers of polynomials in his method and gives the formula \e{bell2} in \cite[pp. 43-44]{Arb}.
For example, in a table on p. 29 of \cite{Arb},   the entry
\begin{equation}\label{xx}
   6\beta^5 \zeta +15\beta^4(2 \g \varepsilon+\delta^2)
+60\beta^3 \g^2 \delta+15 \beta^2 \g^4
\end{equation}
associated to $x^{10}$ and $D^6$ is provided. In our notation this is $\dm_{10,6}(\beta,\gamma,\delta,\varepsilon,\zeta,\dots)$,
or more clearly,
\begin{equation}\label{xx2}
  \dm_{10,6}(a_1, a_2, a_3, \dots) = 6a_1^5 a_5 +30a_1^4 a_2 a_4+15a_1^4 a_3^2
+60a_1^3 a_2^2 a_3 +15 a_1^2 a_2^4.
\end{equation}
We explain what he was doing after Theorem \ref{arb}.

Abraham De Moivre (1667-1754), though perhaps best known today for a trigonometric formula,  was an important pioneer  in areas such as probability and statistics, roots of equations,  analytic geometry and  infinite series. Born in France, he moved to England with other Huguenots to escape  persecution, and there  became part of the circle that included Newton, Halley and Stirling. 
The biography \cite{dbio} gives more details about his life, and the papers \cite{ScIv,Cr05,Gel} describe aspects of his mathematics, some of which we will touch on. 

\section{Basic properties of $\dm_\MakeLowercase{n,k}$} \label{bas2}

The results in this section follow  from the generating function \e{bell} as  exercises. Some of these and their partial Bell polynomial analogs  appear in \cite[pp. 188 -- 192]{Ri68} and \cite[Sec. 3.3]{Comtet}.
 The following equalities \e{esb} -- \e{fsb5} are identities in the ring $\Z[a_1,a_2, \dots]$.

For small $k$,  
\begin{align}\label{esb}
  \dm_{n,0}(a_1, a_2,   \dots) & = \delta_{n,0}, \\
  \dm_{n,1}(a_1, a_2,   \dots) & = a_n \qquad \qquad (n\gqs 1), \label{esb2} \\
  \dm_{n,2}(a_1, a_2,   \dots) & = \sum_{j=1}^{n-1} a_j a_{n-j}. \label{esb3}
\end{align}
The identities \e{esb2}, \e{esb3} generalize to give a symmetric formula for $\dm_{n,k}$:
\begin{equation}\label{pobell3}
  \dm_{n,k}(a_1, a_2,  \dots)  = \sum_{n_1+n_2+\dots + n_k = n}
    a_{n_1}a_{n_2} \cdots a_{n_k} \qquad \qquad (k \gqs 1),
\end{equation}
where the sum  in  \e{pobell3} is over all possible $n_1$, $n_2,  \dots, n_k \in \Z_{\gqs 1}$.
The two main recursion relations for $\dm_{n,k}$ in $n$ and $k$ are given by
\begin{align}\label{rec}
  \dm_{n+k,k}(a_1,a_2,a_3, \dots) & =\sum_{j=0}^k \binom{k}{j} a_1^{k-j} \dm_{n,j}(a_2,a_3,\dots), \\
  \dm_{n,k+1}(a_1,a_2,a_3, \dots) & =\sum_{j=k}^{n-1} a_{n-j} \dm_{j,k}(a_1,a_2,a_3, \dots). \label{rec2}
\end{align}
In fact \e{rec2} is the $\ell=1$ case of the following identity: 
\begin{equation}\label{rec3}
  \sum_{j=0}^n \dm_{j,k}(a_1,a_2, \dots) \cdot \dm_{n-j,\ell}(a_1,a_2, \dots) = \dm_{n,k+\ell}(a_1,a_2, \dots).
\end{equation}
Note that \e{rec} is useful to compute $\dm_{n+k,k}$ if $n$ is small, since the terms in the sum can only be nonzero for $j \lqs n$. For example, when $k\gqs 1$
we have
\begin{gather}
  \dm_{k,k}(a_1, a_2,  \dots)  = a_1^k, \label{fsb}\\
  \dm_{k+1,k}(a_1, a_2, \dots)  = k a_1^{k-1} a_2, \label{fsb2}\\
\dm_{k+2,k}(a_1, a_2,  \dots)  = \binom{k}{1} a_1^{k-1} a_3 + \binom{k}{2} a_1^{k-2} a_2^2 , \label{fsb3}\\
\dm_{k+3,k}(a_1, a_2,  \dots)  =  \binom{k}{1} a_1^{k-1} a_4 +
  2\binom{k}{2} a_1^{k-2} a_2 a_3 + \binom{k}{3} a_1^{k-3} a_2^3, \label{fsb4}\\
\dm_{k+4,k}(a_1,a_1, \dots) = \binom{k}{1} a_1^{k-1} a_5 +
  \binom{k}{2} a_1^{k-2}\left( a_3^2 + 2 a_2 a_4\right) + 3\binom{k}{3} a_1^{k-3} a_2^2 a_3  +
  \binom{k}{4} a_1^{k-4} a_2^4. \label{fsb5}
\end{gather}
 De Moivre continued this list as far as $\dm_{k+6,k}$ in \cite[Fig. 5]{dem}.

 For a variable $z$, the general binomial coefficients satisfy $\binom{z}{0}:= 1$  and
for positive integers $k$,
\begin{equation}\label{bin}
  \binom{z}{k}:= \frac{z(z-1) \cdots (z-k+1)}{k!}, \qquad  \binom{-z}{k}=(-1)^k\binom{z+k-1}{k}.
\end{equation}

\begin{lemma} For $n \gqs k \gqs 0$, we have the following identities in $\Q[z]$:
\begin{align}
\dm_{n,k}\left(1,1,1,\dots  \right) &   =\binom{n-1}{n-k} \label{bb1}, \\
\dm_{n,k}\left( \binom{z}{0}, \binom{z}{1}, \binom{z}{2},\dots  \right) & = \binom{k z}{n-k}, \label{aa2}\\
\dm_{n,k}\left( \binom{z}{0}, \binom{z+1}{1}, \binom{z+2}{2},\dots  \right) & = \binom{n+k z -1}{n-k}. \label{aa3}
\end{align}
\end{lemma}
\begin{proof}
The equality \e{aa2} follows from the binomial theorem.  Then \e{aa2} implies \e{aa3},   using the right identity in \e{bin}. Lastly, \e{bb1} is a special case of \e{aa3}.
\end{proof}

 The next relations are often useful:
\begin{align}\label{gsb}
\dm_{n,k}(0, a_1, a_2, a_3, \dots) & =  \dm_{n-k,k}(a_1, a_2, a_3, \dots), \\
  \dm_{n,k}(c a_1, c a_2, c a_3, \dots) & = c^k \dm_{n,k}(a_1, a_2, a_3, \dots), \label{gsb2}\\
  \dm_{n,k}(c a_1, c^2 a_2, c^3 a_3, \dots) & = c^n \dm_{n,k}(a_1, a_2, a_3, \dots). \label{gsb3}
\end{align}
Equalities \e{bb1} and \e{gsb2} easily give:

\begin{lemma} Suppose the $a_j$ are complex numbers  satisfying $|a_j|\lqs Q$. Then for $n \gqs k \gqs 0$,
\begin{equation*}
  \left|\dm_{n,k}(a_1,a_2, \dots)\right| \lqs \binom{n-1}{n-k}Q^k \lqs 2^{n}Q^k.
\end{equation*}
\end{lemma}

Some further special values of $\dm_{n,k}$ are described next, related to the exponential function and the logarithm.  We have
\begin{equation}\label{aa1}
   \dm_{m+k,k}\left(\frac 1{0!},\frac 1{1!},\frac 1{2!},\dots  \right)  = \frac{k^m}{m!} \qquad (m,k\gqs 0).
\end{equation}
Also, for $n,k\gqs 0$,
\begin{alignat}{3}\label{bakb}
  (e^t-1)^k  & = \sum_{n=0}^\infty \frac{k!}{n!}\stirb{n}{k} t^n & \qquad & \implies \qquad & \dm_{n,k}\left( \frac{1}{1!}, \frac{1}{2!}, \frac{1}{3!},  \dots \right) & = \frac{k!}{n!}\stirb{n}{k},
  \\
  (-\log(1-t))^k & = \sum_{n=0}^\infty \frac{k!}{n!}\stira{n}{k} t^n
 & \qquad & \implies \qquad &\dm_{n,k}\left( \frac{1}{1}, \frac{1}{2}, \frac{1}{3},  \dots \right) & =
  \frac{k!}{n!}\stira{n}{k}. \label{baka}
\end{alignat}
The left identities of \e{bakb} and \e{baka}, (see \cite[(7.49), (7.50)]{Knu}), may be taken as the definitions of the Stirling numbers, as in \cite[p. 51]{Comtet}, and  all their properties developed from this starting point. Alternatively, the Stirling subset numbers $\stirb{n}{k}$ count the number of ways to partition  $n$ elements into $k$ nonempty subsets, and the Stirling cycle numbers $\stira{n}{k}$ count the number of ways to arrange $n$ elements into $k$ cycles. More advanced special values of $\dm_{n,k}$  are shown in \cite[Sect. 9.3]{OSsym}.

We  notice  that the De Moivre polynomials in  \e{bakb} appear with  arguments shifted from the  ones  in \e{aa1}. This type of shifting will also be seen   in  examples in sections \ref{gfn} and  \ref{laps}.
In the simplest case, adding or removing the first coefficient $a_1$ in $\dm_{n,k}$ has a simple effect, by the binomial theorem:
\begin{lemma}
For $k \gqs 0$,
\begin{align}
  \dm_{n,k}(a_2,a_3, \dots) & = \sum_{j=0}^k (-a_1)^{k-j} \binom{k}{j} \dm_{n+j,j}(a_1,a_2, \dots), \label{add}\\
  \dm_{n,k}(a_1,a_1, \dots) & = \sum_{j=0}^k a_1^{k-j} \binom{k}{j} \dm_{n-k,j}(a_2,a_3, \dots). \label{drop}
\end{align}
\end{lemma}

(The identity \e{drop} is just \e{rec} again.) Then applying \e{add} and \e{drop} $r$ times leads to the following.

\begin{prop} \label{shift}
For $r, k \gqs 0$,
$\dm_{n,k}(a_{r+1},a_{r+2}, \dots)$ equals
\begin{equation}\label{fortx}
   \sum_{ j_1+ j_2+ \dots +j_{r+1}= k}
 \binom{k}{j_1 , j_2 ,  \dots , j_{r+1}} (-a_1)^{j_1} (-a_2)^{j_2}  \cdots (-a_r)^{j_r}
 \dm_{n+J+r j_{r+1},j_{r+1}}(a_{1},a_{2}, \dots)
\end{equation}
and
$\dm_{n,k}(a_1,a_2, \dots)$ equals
\begin{equation}\label{fort}
   \sum_{ j_1+ j_2+ \dots +j_{r+1}= k}
 \binom{k}{j_1 , j_2 ,  \dots , j_{r+1}} a_1^{j_1} a_2^{j_2}  \cdots a_r^{j_r}
 \dm_{n+J-rk,j_{r+1}}(a_{r+1},a_{r+2}, \dots)
\end{equation}
where $J$ means $(r-1)j_1+(r-2)j_2+ \cdots +1 j_{r-1}$ and the summations are  over all  $j_1$,  \dots , $j_{r+1} \in \Z_{\gqs 0}$ with sum $k$.
\end{prop}

For example, taking $r=n-k+1$ and setting $a_{r+1}=a_{r+2}= \cdots=0$ in \e{fort} recovers \e{bell2} since the only nonzero terms have $n+J-rk=j_{r+1}=0$ by \e{esb}.

\section{Manipulating power series} \label{manx}
Our power series coefficients $a_j$ may come from a ring $R$.  Though more general cases can be considered, a natural choice  for $R$ is an  integral domain containing $\Z$ (and so of characteristic $0$). Throughout this section we assume $R$ has these properties and hence $R[[x]]$, the ring of formal power series over $R$, is also an integral domain containing $\Z$. See  \cite[Chap. 1]{GJ83}, \cite[Sect. 1.12]{Comtet} or \cite[Sect. A5]{Fl09} for more on  formal power series rings. An important point is that, at each step, a new coefficient must depend only on finitely many others, as if we were working with polynomials; see the notions of summable and admissible in \cite[Chap. 1]{GJ83}. This is why the inner series in Proposition \ref{comp} cannot have a constant term, for example.


\subsection{Series composition}
\begin{prop} \label{comp}
Suppose that $f(x)=a_1 x+a_2 x^2+ \cdots$ and $g(x)=b_0+b_1 x+b_2 x^2+ \cdots$ are two power series in $R[[x]]$. Then
$
  g(f(x))=c_0+c_1 x+c_2 x^2+ \cdots $ is in $R[[x]]$
with
\begin{equation}\label{esto}
 \quad c_n = \sum_{k=0}^n b_k  \cdot \dm_{n,k}(a_1,a_2,\dots).
\end{equation}
\end{prop}
\begin{proof}
We have
$$
g(f(x))=\sum_{k=0}^\infty b_k \cdot f(x)^k = \sum_{k=0}^\infty b_k \sum_{n=k}^\infty \dm_{n,k}(a_1,a_2,\dots) x^n,
$$
and switching the order of summation gives \e{esto}.
\end{proof}

Though this result is for formal power series, it applies very widely. A commonly occurring situation has $f(x)$ and $g(x)$ holomorphic in neighborhoods of $0$ with $f(0)=0$. Then $g(f(x)$ is also holomorphic in a neighborhood of $0$ and its Taylor coefficients may be computed using Proposition \ref{comp}.

Three important cases  are as follows. Suppose that the integral domain $R$ contains $\Q$.
Applying Proposition \ref{comp} with $g(x)=(1+x)^\alpha$, $e^{\alpha x}$ and $\log(1+\alpha x)$ gives
\begin{align}\label{pot}
  [x^n] (1+f(x))^\alpha & = \sum_{k=0}^n \binom{\alpha}{k}  \dm_{n,k}(a_1,a_2,\dots),\\
  [x^n] e^{\alpha f(x)} & = \sum_{k=0}^n \frac {\alpha^k}{k!} \dm_{n,k}(a_1,a_2,\dots),\label{cox}\\
  [x^n] \log(1+ \alpha f(x)) & = \sum_{k=1}^n (-1)^{k-1}\frac{\alpha^k}{k} \dm_{n,k}(a_1,a_2,\dots). \label{log}
\end{align}
The polynomials in $a_1, a_2, \dots$ on the right sides of \e{pot}, \e{cox} and \e{log} are essentially the potential, complete exponential and logarithmic polynomials, respectively, of \cite[Sections 3.3, 3.5]{Comtet}.
Recalling \e{bin},  the right sides of \e{pot}, \e{cox} and \e{log} are also degree $n$ polynomials in $\alpha$.
Therefore the series $(1+f(x))^\alpha$, $e^{\alpha f(x)}$ and $\log(1+ \alpha f(x))$ make sense for arbitrary $\alpha$,  giving elements of $R[\alpha][[x]]$.
Further, if $f(x)$ is holomorphic in a neighborhood of $0$, then so are the compositions $g(f(x))$, with \e{pot}, \e{cox}, \e{log} (times $n!$) giving their Taylor coefficients.

For an example  we will need  later, consider
\begin{equation} \label{hot}
   \log\left( 1+\frac{\log(1+x)}{u}\right) = \sum_{n=1}^\infty \ell_n(u) x^n,
\end{equation}
and we would like to know how $\ell_n(u)$ depends on $u$.
Stepping through the   proof of Proposition \ref{comp} shows
\begin{align*}
    \sum_{k=1}^\infty \frac{(-1)^{k+1}}{k \cdot u^k} \log^k(1+x)
  & = \sum_{k=1}^\infty \frac{(-1)^{k+1}}{k \cdot u^k} \sum_{n=k}^\infty \dm_{n,k}( 1,-{\textstyle \frac 12},{\textstyle \frac 13},\dots) x^n \\
& = \sum_{n=1}^\infty x^n  \sum_{k=1}^n  \frac{(-1)^{k+1}}{k \cdot u^k} \dm_{n,k}( 1,-{\textstyle \frac 12},{\textstyle \frac 13},\dots).
\end{align*}
It is already clear that $\ell_n(u)$ is a polynomial of degree $n$ in $1/u$ with no constant term. With \e{gsb2}, \e{gsb3} and \e{baka} we obtain the simplification
\begin{equation}\label{lnu}
  \ell_n(u) = (-1)^{n+1} \sum_{k=1}^n \frac{(k-1)!}{n!} \stira{n}{k} u^{-k}.
\end{equation}

A very common case that follows from \e{pot} should be  highlighted:

\begin{prop} \label{mul}
Let $R$ be an integral domain containing $\Z$ with $a_0 + a_1 x +a_2 x^2+ \cdots$  in $R[[x]]$.
Then for $m\in \Z$,
\begin{equation}\label{reab}
  \left(a_0 + a_1 x +a_2 x^2+  \cdots \right)^{m} = \sum_{n=0}^\infty
\left[\sum_{k=0}^n \binom{m}{k} a_0^{m-k} \dm_{n,k}(a_1, a_2,  \dots) \right] x^n
\end{equation}
is in  $R[[x]]$, (provided $a_0$ is invertible in $R$ if $m<0$), and in particular, the multiplicative inverse of the series  is given by
\begin{equation}\label{recip}
  \left(a_0 + a_1 x +a_2 x^2+  \cdots \right)^{-1} = \sum_{n=0}^\infty
\left[\sum_{k=0}^n (-1)^k a_0^{-k-1} \dm_{n,k}(a_1, a_2,  \dots) \right] x^n.
\end{equation}
\end{prop}

\subsection{Arbogast's formula}

See  \cite{Cr05} and \cite{Jo02} for the early history of the next famous result which is usually named for Fa\`{a} di Bruno. 
It is essentially equivalent to Proposition \ref{comp}.

\begin{theorem}[Arbogast's formula] \label{arb}
For $n$ times differentiable functions $f$ and $g$,
\begin{equation}\label{faa}
 \frac{1}{n!} \frac{d^n}{dx^n} g(f(x)) = \sum_{k=0}^n \frac{g^{(k)}(f(x))}{k!} \cdot \dm_{n,k}\left( \frac{f'(x)}{1!},\frac{f''(x)}{2!}, \dots  \right).
\end{equation}
\end{theorem}

\begin{proof}
Use induction to establish that \e{faa} must be true for some polynomial $P_{n,k}$ instead of $\dm_{n,k}$. To identify $P_{n,k}$
set $f(x)=a_1 x+a_2 x^2+ \cdots$, $g(x)=b_0+b_1 x+b_2 x^2+ \cdots$ and $g(f(x))=c_0+c_1 x+c_2 x^2+ \cdots$.
Then evaluating   at  $x=0$ finds
\begin{equation*}
   c_n = \sum_{k=0}^n b_k \cdot P_{n,k}\left( a_1,a_2, \dots  \right).
\end{equation*}
Comparing this with \e{esto} now shows that  $P_{n,k} \equiv \dm_{n,k}$, since the coefficients $a_i$ and $b_j$  are arbitrary.
\end{proof}

Arbogast had a similar point of view in \cite{Arb}. He considered
 $f(x)=a_0+a_1 x +a_2 x^2+  \cdots $ and gave the formula
\begin{equation} \label{arbo}
  \frac{1}{n!}\left. \frac{d^n}{dx^n} g(f(x))\right|_{x=0} = \sum_{k=0}^n \frac{g^{(k)}(a_0)}{k!} \cdot \dm_{n,k}\left( a_1,a_2, \dots  \right),
\end{equation}
computing \e{arbo} explicitly in tables for $1\lqs n \lqs 10$. For example, see Figure 1 of \cite{Cr05}, reproduced from \cite[p. 29]{Arb}, where
 $\dm_{10,k}(\beta,\gamma,\delta,\varepsilon,\zeta,\nu,\theta,\iota,\kappa,\lambda,\dots)$ is given correctly for $1\lqs k \lqs 10$.

In the notation of Proposition \ref{comp}, a further composition $h(g(f(x)))$  would require $\dm_{n,r}(c_0,c_1, \dots)$. Computing this using \e{esto} looks hopelessly complicated, with De Moivre polynomials inside De Moivre polynomials, but in fact a slight extension of a lemma  of  Charalambides, \cite[Lemma 11.1]{chacha},  gives a simple formula.

\begin{prop} \label{comp2}
With the assumptions of Proposition \ref{comp},
\begin{equation}\label{chch}
\dm_{n+r,r}(c_0, c_1, \dots ) = \sum_{k=0}^n   \dm_{n,k}(a_1,a_2,\dots)\cdot \dm_{k+r,r}(b_0,b_1,\dots).
\end{equation}
\end{prop}
\begin{proof}
Write $y=f(x)$ so that
\begin{equation*}
  (y \cdot g(y))^r = \sum_{k=r}^\infty \dm_{k,r}(b_0,b_1,\dots) y^k,
\end{equation*}
and hence
\begin{align*}
g(y)^r & =\sum_{k=0}^\infty \dm_{k+r,r}(b_0,b_1,\dots) y^k \\
& = \sum_{k=0}^\infty \dm_{k+r,r}(b_0,b_1,\dots)
 \sum_{n=k}^\infty \dm_{n,k}(a_1,a_2,\dots) x^n\\
 & = \sum_{n=0}^\infty x^n \sum_{k=0}^n   \dm_{n,k}(a_1,a_2,\dots)\cdot \dm_{k+r,r}(b_0,b_1,\dots).
\end{align*}
To finishes the proof, note that
\begin{equation*}
  \dm_{n+r,r}(c_0, c_1, \dots ) = [x^{n+r}](c_0x+c_1 x^2+ \cdots)^r = [x^n]g(y)^r. \qedhere
\end{equation*}
\end{proof}

(Use \e{gsb} to simplify this result if $b_0=c_0=0$.) Now we see that Proposition \ref{comp} is just the $r=1$ case of Proposition \ref{comp2}. Also \e{chch} suggests the following natural extension of Arbogast's formula
 to  $r$th powers of a composition. It is probably contained in the large literature on Arbogast's formula, but we have not  found it.
 
 \begin{theorem} \label{lar}
Let $n$ and $r$ be nonnegative integers. For $n$ times differentiable functions $f$ and $g$,
\begin{equation}\label{faa22}
 \frac{1}{n!} \frac{d^n}{dx^n} g(f(x))^r =\sum_{k=0}^n  \dm_{n,k}\left( \frac{f'(x)}{1!},\frac{f''(x)}{2!}, \dots  \right) \cdot \dm_{k+r,r}\left( \frac{g(f(x))}{0!},\frac{g'(f(x))}{1!}, \dots  \right).
\end{equation}
\end{theorem}
\begin{proof}
Let $G(x):=g(x)^r$. Applying Proposition \ref{comp} finds
\begin{equation}\label{faax}
 \frac{1}{n!} \frac{d^n}{dx^n} G(f(x)) = \sum_{k=0}^n \dm_{n,k}\left( \frac{f'(x)}{1!},\frac{f''(x)}{2!}, \dots  \right) \cdot \frac{G^{(k)}(f(x))}{k!}.
\end{equation}
With $G(x)=h(g(x))$ for $h(x):=x^r$, another application gives
\begin{align}
  \frac{G^{(k)}(y)}{k!} & = \sum_{j=0}^k \frac{h^{(j)}(g(y))}{j!} \cdot \dm_{k,j}\left( \frac{g'(y)}{1!},\frac{g''(y)}{2!}, \dots  \right) \notag\\
  & = \sum_{j=0}^k \binom{r}{j} g(y)^{r-j} \cdot \dm_{k,j}\left( \frac{g'(y)}{1!},\frac{g''(y)}{2!}, \dots  \right). \label{jean}
\end{align}
Since $\binom{r}{j}=0$ for $j>r$ and $\dm_{k,j}=0$ for $j>k$, we may change the upper limit of summation in \e{jean} from $k$ to $r$. Then by \e{rec},
\begin{equation*}
  \frac{G^{(k)}(y)}{k!} = \dm_{k+r,r}\left( g(y),\frac{g'(y)}{1!}, \dots  \right),
\end{equation*}
and the desired formula follows.
\end{proof}

Lastly we mention that it is natural to  generalize \e{bell} to power series of more than one variable and then consider their compositions. This results in De Moivre polynomials with  elaborate indexing; see the recent paper \cite{Sch} and its contained references for details.

\subsection{Compositional inverses}
We call $g$ a compositional inverse of $f$ if $g(f(x))=x=f(g(x))$.
The usual proof of the following standard result becomes especially clear by employing De Moivre polynomials.

\begin{prop}
The series $f(x)=a_1 x+a_2 x^2+ \cdots$ in $R[[x]]$ has a compositional inverse in $R[[x]]$ if and only if $a_1$ is invertible in $R$. This inverse of $f(x)$ is unique.
\end{prop}
\begin{proof}
Suppose $a_1$ is invertible. To find the compositional inverse, we assume it has the form $g(x)=b_1 x+b_2 x^2+ \cdots$ and try to solve for the coefficients using \e{esto}.
We want $$1=c_1= b_1  \cdot \dm_{1,1}(a_1,a_2,\dots)=b_1 a_1$$ and  so $b_1=1/a_1 \in R$. If we assume that we have found $b_1$, $b_2, \dots, b_{m-1}$ in $R$ already, then $b_m$ is the solution to
\begin{equation}\label{recfaa}
0 = c_m = \sum_{k=1}^{m-1} b_k  \cdot \dm_{m,k}(a_1,a_2,\dots) + b_m  \cdot \dm_{m,m}(a_1,a_2,\dots).
\end{equation}
With \e{fsb},   $\dm_{m,m}(a_1,a_2,\dots) = a_1^m$. Since this  is again invertible, we see that $b_m \in R$. By induction we obtain  $g$ in $R[[x]]$ so that $g(f(x))=x$. Repeating this argument for  $g$ shows that there exists $h$ in $R[[x]]$ with $h(g(x))=x$. Then
$$
f(g(x))=h(g(f(g(x))))=h(g(x))=x.
$$
Therefore $g$ is a compositional inverse of $f$. Uniqueness is proved in the usual way for inverses.

In the other direction, if $f$ has a compositional inverse $g(x)=b_1 x+b_2 x^2+ \cdots$ in $R[[x]]$ then we must have $a_1 b_1 = 1$ and so $a_1$ is  invertible.
\end{proof}

The relation \e{recfaa} gives a recursive procedure to find the coefficients of the inverse.  There is a more direct way to find them though, using Lagrange inversion.

\begin{theorem}[Lagrange inversion for ${R[[x]]}$] \label{xlinv}
Suppose  $f(x)$ and $g(x)$ are compositional inverses in $R[[x]]$, with neither having a constant term.  Then
\begin{equation}\label{xlagr}
  m[x^m] g(x)^r = r[x^{m-r}](x/f(x))^m \qquad (m, r \in \Z_{\gqs 1}).
\end{equation}
\end{theorem}

See \cite[Sec. 1.2]{GJ83} or \cite{Ge16} for this result and generalizations. A simple proof involves formal Laurent series and their residues, meaning the coefficients of $x^{-1}$. Note that the right side of \e{xlagr} is in $R[[x]]$ by Proposition \ref{mul},
and using its expansion gives:

\begin{cor} \label{cinv}
Suppose  $f(x)=a_1 x+a_2 x^2+ \cdots$ in $R[[x]]$ has  compositional inverse $b_1 x+b_2 x^2+ \cdots$. Then for positive integers $m$, $r$,
\begin{align}\label{lag}
  b_m = \frac 1{m} \sum_{k=0}^{m-1} \binom{-m}{k} a_1^{-m-k} \dm_{m-1,k}(a_2,a_3,\dots),\\
  \dm_{m,r}(b_1,b_2, \dots) = \frac r{m} \sum_{k=0}^{m-r} \binom{-m}{k} a_1^{-m-k} \dm_{m-r,k}(a_2,a_3,\dots).
  \label{lagxx}
\end{align}
\end{cor}

Recall that $R$ does not necessarily contain $\Q$. The $1/m$ terms on the right of \e{lag} and  \e{lagxx} must cancel an $m$ factor so that these are statements in $R$. We see this cancellation explicitly in section \ref{divv}.

The next result follows from \e{lag}, in the same way that Theorem \ref{arb} followed from Proposition \ref{comp}, and easily extends to  derivatives of $g(x)^r$. See \cite{Jo02b} for another approach and references.

\begin{cor}
Let $f(x)$ be an $m \gqs 1$ times differentiable function with compositional inverse $g(x)$. Then $g(x)$ is $m$ times differentiable and
\begin{equation} \label{laginv2}
   \frac{1}{m!} \frac{d^m}{dx^m} g(x) = \frac{1}{m}\sum_{k=0}^{m-1} \binom{-m}{k}\frac{1}{f'(g(x))^{m+k}}  \dm_{m-1,k}\left( \frac{f''(g(x))}{2!},\frac{f'''(g(x))}{3!}, \dots  \right).
\end{equation}
\end{cor}

A different kind of inversion was needed in \cite[Eqs. (11), (14)]{Moyal}:

\begin{prop}
Let $a_1,a_2,\dots$ be a sequence of complex numbers and set
\begin{equation*}
  b_n(t):= \sum_{k=1}^n \frac{t^k}{k!} \dm_{n,k}(a_1,a_2,\dots) \qquad (n\gqs 1).
\end{equation*}
Then  these relations  may be inverted:
\begin{equation} \label{invv}
  a_n = \frac 1t \sum_{k=1}^n \frac{(-1)^{k+1}}{k} \dm_{n,k}(b_1(t),b_2(t),\dots) \qquad (n\gqs 1).
\end{equation}
\end{prop}
\begin{proof} By Proposition \ref{comp},
 $\exp(t(a_1 x+a_2 x^2+\cdots)) = 1+b_1(t)x+b_2(t)x^2+\cdots$. Then take the log of both sides  to obtain \e{invv}.
\end{proof}

Of course, $\exp(x)$ and $\log(x)$ above may be replaced by other pairs of inverse functions; see \e{tpx} and \e{ptx} for  a different pair. More complicated inversions of this type appear in \cite{BGW}.

\section{Greatest common divisor of the coefficients of $\dm_{n,k}(x_1,x_2,\dots)$} \label{divv}

It follows from Corollary \ref{cinv},
when $a_1=1$ and $R=\Z$, that for integers $n,m,k \gqs 1$ we must have
\begin{equation} \label{dog}
  \frac n{m}  \binom{m}{k}  \dm_{n,k}(x_1,x_2,\dots) \in \Z[x_1,x_2,\dots].
\end{equation}
For example, with $n=9$ and $k=4$ this means that all the coefficients of $\dm_{9,4}$ must be divisible by $4$. The next result, combined with \e{arco2} below, may be used to prove \e{dog} directly.

\begin{theorem} \label{argcd}
For $n\gqs k \gqs 1$,  the $\gcd$ 
 of the coefficients of $\dm_{n,k}$ is   $k/\gcd(n,k)$.
\end{theorem}
\begin{proof} Let $\Delta$ be the $\gcd$ 
 of the coefficients of $\dm_{n,k}$. We may assume that $n>k$ as the theorem is clearly true when $n=k$ by \e{fsb}.
Recall that the coefficients of $\dm_{n,k}$ are the multinomial coefficients
\begin{equation}\label{arco}
  \binom{k}{j_1 , j_2 ,  \dots , j_m} \quad \text{with} \quad 1j_1+2 j_2+ \dots +mj_m= n.
\end{equation}
Taking $j_1 = k-1$ and $j_{n-k+1}=1$ gives the coefficient $k$ and so $\Delta \mid k$. Our next goal is to show that $(k/\gcd(n,k)) \mid \Delta$.

The elegant argument in \cite{GS95} for binomial coefficients adapts nicely to multinomial coefficients as follows.
It is evident that when
\begin{equation*}
  C:= \binom{k}{j_1 , j_2 ,  \dots , j_m} \quad \text{we have} \quad \frac{j_r}{k} C= \binom{k-1}{j_1 ,  \dots ,j_r-1 ,  \dots , j_m}
\end{equation*}
for $1\lqs r\lqs m$ with $j_r \neq 0$.
Therefore $C_r:= j_r C/k$ is always an integer for $j_r \gqs 0$ and so
\begin{align*}
  k \cdot \gcd\left( C,C_1,C_2, \dots ,C_m\right)
& = \gcd\left( k C,k C_1,k C_2, \dots ,k C_m\right) \\
& = \gcd\left( k C, j_1 C, j_2 C, \dots, j_m C\right) \\
& =  C\gcd(k,j_1,j_2,\dots,j_m).
\end{align*}
Hence, for any multinomial coefficient,
\begin{equation}\label{arco2}
  \frac k{\gcd(k,j_1,j_2,\dots,j_m)}  \quad \text{divides} \quad  \binom{k}{j_1 , j_2 ,  \dots , j_m}.
\end{equation}
Now $\gcd(j_1,j_2,\dots,j_m) \mid n$ by the right side of \e{arco} and it follows that $k/\gcd(n,k)$ divides all the coefficients of $\dm_{n,k}$.

So far we have demonstrated that $(k/\gcd(n,k)) \mid \Delta$ and $\Delta \mid k$. The next lemma will let us find coefficients that ensure $ \Delta \mid (k/\gcd(n,k))$.

\begin{lemma} \label{tro}
Let $\alpha$ and $\beta$ be  positive integers. Then
\begin{equation} \label{chil}
  \gcd\left(\binom{\alpha \beta}{d_1}, \binom{\alpha \beta}{d_2}, \dots , \binom{\alpha \beta}{d_r} \right)  \quad \text{divides} \quad  \alpha
\end{equation}
where $d_1,d_2, \dots,d_r$ are all the divisors of $\beta$.
\end{lemma}
\begin{proof}
For $p$  any prime, let $\nu_p$ denote the $p$-adic valuation. Set $x:=\nu_p(\alpha)$, $y:=\nu_p(\beta)$ so that $\alpha=\alpha' p^x$ and $\beta=\beta' p^y$ for $\alpha'$, $\beta'$ prime to $p$. We want to establish
\begin{equation} \label{chil2}
  \nu_p \left( \binom{\alpha \beta}{p^y}\right) = x,
\end{equation}
since that will show that the largest power of $p$ dividing the left side of \e{chil} is the power of $p$ that divides $\alpha$.

Use the notation $s_p(n)$ to mean the sum of the base $p$ digits of $n$. Legendre's  well-known formula implies, as in \cite[Eq. (1.6)]{St09},
\begin{equation*}
  \nu_p \left( \binom{\alpha \beta}{p^y}\right) = \frac{s_p(p^y)+s_p(\alpha \beta-p^y)-s_p(\alpha \beta)}{p-1}.
\end{equation*}
The numerator on the right equals
\begin{equation*}
  s_p(p^y)+s_p((\alpha'\beta'p^x-1)p^y)-s_p(\alpha'\beta'p^{x+y}) = 1+s_p(\alpha'\beta'p^x-1)-s_p(\alpha'\beta').
\end{equation*}
If $\alpha'\beta'$ has the base $p$ representation $c_1c_2 \cdots c_k$ then $c_k \neq 0$ and $\alpha'\beta'p^x$ has the same digits with $x$ zeros added on the right. Clearly $\alpha'\beta'p^x -1$ will have the  representation
$$
  c_1c_2 \cdots (c_k-1)(p-1) \cdots (p-1),
$$
so that $s_p(\alpha'\beta'p^x -1) = s_p(\alpha'\beta')-1+x(p-1)$. Then \e{chil2} follows and the lemma is proved.
\end{proof}

Now let $d$ be any divisor of $\gcd(n,k)$. We consider coefficients in \e{arco} where only two numbers $j_1$ and $j_r$ are nonzero, choosing $r:=(n-k)/d+1>1$. Check that $j_1=k-d$ and $j_r=d$ gives $1j_1+rj_r=n$ and so
\begin{equation}\label{wind}
  \binom{k}{d}   \quad \text{for} \quad d \mid \gcd(n,k)
\end{equation}
 is a coefficient of $\dm_{n,k}$. Apply Lemma \ref{tro} with $\alpha=k/\gcd(n,k)$ and $\beta=\gcd(n,k)$ to see that the $\gcd$ of the coefficients of the form \e{wind} must divide $k/\gcd(n,k)$. Therefore $ \Delta \mid (k/\gcd(n,k))$ as we wanted. This finishes the proof of Theorem \ref{argcd}.
\end{proof}

By comparison, the $\gcd$s of the coefficients of the partial Bell polynomials are much larger.  It may be shown  for example that  $\bl_{k+3,k}$ has $\gcd$ equalling $\binom{k+3}{4}$.

\section{Determinant formulas} \label{dets}

Define the $n \times n$ matrix
\begin{equation*}
  \mathcal{M}_n(t) := \left(
  \begin{matrix}
  a_1 t & 1 & &  \\
  a_2 t & a_1 t & 1 &  \\
  \vdots & & \ddots &  \\
  a_n t & a_{n-1} t & \dots &  a_1 t
  \end{matrix}
  \right),
\end{equation*}
with ones above the main diagonal and zeros above those.
The next result is a slightly simplified version of \cite[Thm. 3.1]{EJ}.

\begin{prop} \label{detp}
We have
\begin{equation} \label{far}
  \sum_{k=0}^n (-1)^{n+k} t^k \dm_{n,k}(a_1,a_2, \dots) = \det \mathcal{M}_n(t).
\end{equation}
\end{prop}
\begin{proof}
By \e{gsb2}, $t^k \dm_{n,k}(a_1,a_2, \dots)= \dm_{n,k}(a_1 t,a_2 t, \dots)$ and so it is enough to prove \e{far} for $t=1$. An induction proof is possible since repeatedly expanding the determinant along the top row gives
\begin{equation*}
  \det \mathcal{M}_n(1) = a_1 \det \mathcal{M}_{n-1}(1) - a_2  \det \mathcal{M}_{n-2}(1)+ \cdots + (-1)^{n-1} a_n.
\end{equation*}
A more illuminating proof uses \e{recip}. For $u_n:= \sum_{k=0}^n (-1)^{k} \dm_{n,k}(a_1,a_2, \dots)$ we obtain
$$
(1+a_1 x+a_2 x^2+ \cdots)(1+u_1 x+u_2 x^2+ \cdots) =1.
$$
Therefore
\begin{equation}\label{mat}
\left(
  \begin{matrix}
  1  &  & &  & \\
  a_1  & 1  &  &  & \\
  a_2  & a_1  & 1 & &  \\
  \vdots & & \ddots &  & \\
  a_{n-1}  & a_{n-2}  & \dots & a_1 & 1
  \end{matrix}
  \right)
  \left(
  \begin{matrix}
  u_1  \\
  u_2 \\
  u_3  \\
  \vdots   \\
  u_n
  \end{matrix}
  \right)
  =
  \left(
  \begin{matrix}
  -a_1  \\
  -a_2 \\
  -a_3  \\
  \vdots   \\
  -a_n
  \end{matrix}
  \right)
\end{equation}
and, by Cramer's rule, $u_n=\det U_n/\det U$ where $U$ is the matrix on the left of \e{mat} and $U_n$ is $U$ with its $n$th column replaced by the column on the right side of \e{mat}. Moving this $n$th column of $U_n$ to the left side and changing its sign introduces a $(-1)^n$ factor. 
\end{proof}

The $\dm_{n,k}$ polynomials can be isolated in \e{far} since clearly
\begin{equation*}
  \dm_{n,k}(a_1,a_2, \dots) = (-1)^{n+k} [t^k] \det \mathcal{M}_n(t).
\end{equation*}
We may also give exponential and logarithmic versions of Proposition \ref{detp}.
Define the matrices
\begin{equation*}
  \mathcal{N}_n(t) := \left(
  \begin{matrix}
  a_1 t & 1 & & & \\
  a_2 t & a_1 t & 2 & & \\
  a_3 t & a_2 t & a_1 t & 3 & \\
  \vdots & & & \ddots &  \\
  a_n t & a_{n-1} t & \dots & a_2 t & a_1 t
  \end{matrix}
  \right),
  \qquad
  \mathcal{O}_n(t) := \left(
  \begin{matrix}
   a_1 t & 1 & & & \\
  2a_2 t & a_1 t & 1 & & \\
  3a_3 t & a_2 t & a_1 t & 1 & \\
  \vdots & & & \ddots &  \\
  n a_n t & a_{n-1} t & \dots & a_2 t & a_1 t
  \end{matrix}
  \right),
\end{equation*}
which are the same as $\mathcal{M}_n(t)$ except for multiplying by a factor $j$ on row $j$, above the main diagonal in $\mathcal{N}_n(t)$ and in the first column for $\mathcal{O}_n(t)$.

\begin{prop} \label{detp2}
We have
\begin{align} \label{far2}
  \sum_{k=0}^n (-1)^{n+k} \frac{t^k}{k!} \dm_{n,k}\left(\frac{a_1}1,\frac{a_2}2, \dots\right) & = \frac{1}{n!}\det \mathcal{N}_n(t),\\
  \sum_{k=0}^n (-1)^{n+k} \frac{t^k}{k} \dm_{n,k}\left(a_1,a_2, \dots\right) & = \frac{1}{n}\det \mathcal{O}_n(t).
  \label{far3}
\end{align}
\end{prop}
\begin{proof}
To prove \e{far2} we use \e{cox} and it is convenient to set $\alpha_j:=-a_j/j$ and $f(x)=\alpha_1 x+\alpha_2 x^2+ \cdots$. For $v_n:= \sum_{k=0}^n \frac{1}{k!} \dm_{n,k}(\alpha_1,\alpha_2, \dots)$ we obtain
$$
e^{f(x)} = \sum_{n=0}^\infty v_n x^n, \qquad e^{f(x)}f'(x) = \sum_{n=1}^\infty n \cdot v_n x^{n-1}.
$$
 It follows that
$$
(1+v_1 x+v_2 x^2+ \cdots)(\alpha_1+2\alpha_2 x+3\alpha_3 x^2+ \cdots) =(v_1+2v_2 x+3v_3 x^2+ \cdots),
$$
and hence
\begin{equation}\label{mat2}
\left(
  \begin{matrix}
  1  &  & &  & \\
  a_1  & 2  &  &  & \\
  a_2  & a_1  & 3 & &  \\
  \vdots & & \ddots &  & \\
  a_{n-1}  & a_{n-2}  & \dots & a_1 & n
  \end{matrix}
  \right)
  \left(
  \begin{matrix}
  v_1  \\
  v_2 \\
  v_3  \\
  \vdots   \\
  v_n
  \end{matrix}
  \right)
  =
  \left(
  \begin{matrix}
  -a_1  \\
  -a_2 \\
  -a_3  \\
  \vdots   \\
  -a_n
  \end{matrix}
  \right).
\end{equation}
The proof of \e{far2} is now completed using the same steps as in Proposition \ref{detp}. The proof of \e{far3} is similar, based on \e{log}.
\end{proof}

\section{Generating functions} \label{gfn}


Some generating functions that lead to interesting De Moivre expressions are shown next. See for example \cite[Sect. 1.14]{Comtet} and \cite[Chapters 6, 7]{Knu} for more on generating functions.


\subsection{Partitions}
A partition of a positive integer $n$ is a non-increasing sequence of positive integers that sum to $n$. If $p_k(n)$ denotes the number of partitions of $n$ with at most $k$ parts,   then clearly
\begin{equation} \label{pnk}
  p_k(n) = \sum_{\substack{1j_1+2 j_2+ \dots +nj_n= n \\ j_1+ j_2+ \dots +j_n \lqs k}}
 1
\end{equation}
where  the sum   is over all possible $j_1$, $j_2$,  \dots , $j_n \in \Z_{\gqs 0}$ and $j_r$ indicates the number of parts of size $r$.
Comparing \e{pnk} with \e{bell} confirms the statement in Definition \ref{dbf} that the number of terms in $\dm_{n,k}$ is the number of partitions of $n$ with  exactly $k$ parts. This is $p_k(n)-p_{k-1}(n)$ and equals $p_k(n-k)$ as seen from
the generating function
\begin{equation}\label{pkn}
\sum_{n=0}^\infty p_k(n) q^n =  \prod_{j=1}^k \frac 1{1-q^j}.
\end{equation}
Perhaps the most elegant formulas for $p_k(n)$ use Sylvester's theory of waves from 1857 and quasipolynomials, as described in \cite{OS18b}. With \e{recip} we can produce the less enlightening
\begin{equation*}
  p_4(n)= \sum_{k=0}^n \dm_{n,k}(1, 1, 0, 0, -2, 0, 0, 1, 1, -1, \dots),
\end{equation*}
for example, where the rest of the sequence consists of zeros.

Let $p(n)$ denote the number of unrestricted partitions of $n$. Their generating function is \e{pkn} with $k=\infty$. As we saw above, the number of terms in $\dm_{m+k,k}$ is $p_k(m)$. This equals $p(m)$ for $k\gqs m$ and the terms correspond to solutions of $1j_2+2j_3+ \cdots +m j_{m+1}=m$ in \e{bell2}. De Moivre's description in \cite{dem} essentially gives what must be one of the earliest recursive constructions of the partitions of an integer $m$: they are $1$ followed by the partitions of $m-1$, then $2$ followed by the partitions of $m-2$ with smallest part at least $2$, then $3$ followed by the partitions of $m-3$ with smallest part at least $3$, and so on as far as $\lfloor m/2 \rfloor$. Lastly include $m$.


To find a first expression for $p(n)$ we may use Euler's pentagonal number theorem,
\begin{equation*}
  \prod_{j=1}^\infty (1-q^j) = \sum_{r\in \Z} (-1)^r q^{r(3r-1)/2} = \sum_{m=0}^\infty c_m q^m,
\end{equation*}
with $c_m=(-1)^r$ if $m=r(3r-1)/2$ for some $r\in \Z$ and $c_m=0$ otherwise. Therefore, employing \e{recip} again,
\begin{equation*}
  p(n)=\sum_{k=0}^n \dm_{n,k}(-c_1,-c_2,-c_3, \dots) = \sum_{k=0}^n \dm_{n,k}(1,1,0,0,-1,0,-1, \dots).
\end{equation*}

Ramanujan's tau function has the related generating function
\begin{equation} \label{tau}
  \sum_{n=1}^\infty \tau(n) q^n = q \prod_{j=1}^\infty (1-q^j)^{24}
\end{equation}
where the right side of \e{tau} is a modular form when $q=e^{2\pi i z}$, (a weight $12$ cusp form for $\SL_2(\Z)$). Then easily,
\begin{equation} \label{tau2}
  \tau(n)  = \dm_{n+23,24}(c_0,c_1,c_2,\dots) = \dm_{n+23,24}(1,-1,-1,0,0,1,0,\dots)
\end{equation}
and  Theorem \ref{argcd} implies that $\tau(n)$ is divisible by $24/\gcd(n+23,24)$. The congruence properties of $\tau(n)$ and $p(n)$ have been of great interest, and Lehmer's 1947 question as to whether $\tau(n)$ is ever zero remains unanswered.
By \e{pot} we also have the relations
\begin{align} \label{tpx}
\tau(n) & =  \sum_{k=0}^{n-1} \binom{-24}{k} \dm_{n-1,k}(p(1),p(2),p(3),\dots),\\
  p(n) & =  \sum_{k=0}^{n} \binom{-1/24}{k} \dm_{n,k}(\tau(2),\tau(3),\tau(4),\dots).  \label{ptx}
\end{align}

Set $\sigma(n):=\sum_{d | n} d$, the sum of the divisors of $n$. For another  example, take the log of the right side of \e{pkn} when $k=\infty$, expand this into a series and then exponentiate to produce 
\begin{equation} \label{sig}
  p(n)=\sum_{k=0}^n \frac{1}{k!}\dm_{n,k}\left(\frac{\sigma(1)}{1},\frac{\sigma(2)}{2},\frac{\sigma(3)}{3}, \dots \right),
\end{equation}
which may also be inverted. The identity \e{sig} comes from \cite[p. 185]{Ri68}, and  for further identities along these lines see \cite{Jha}.

The authors in \cite[Thm. 3]{BKO} develop a universal recursion for the Fourier coefficients of modular forms. The required polynomials  are in fact De Moivre polynomials, and as a special case (see their remark after the theorem) we find:
\begin{equation}\label{brun}
  \tau(n+1) = -\frac{24\sigma(n)}{n}+\sum_{k=2}^n \frac{(-1)^k}{k} \dm_{n,k}(\tau(2),\tau(3),\dots) \qquad (n\gqs 1).
\end{equation}
The form of \e{brun} now suggests it has an easy proof: take the log of both sides of \e{tau} and expand each series.

\subsection{Orthogonal polynomials}

Bell in \cite{Bell34} was initially interested in generalizing the  Hermite polynomials $H_n(x)$. They may be expressed as $n! [t^n]\exp(2xt-t^2)$ and in Bell's original polynomials equal $\by_n(2x,-2,0,0,\dots)$. In our notation
\begin{equation}\label{her}
  H_n(x) = \sum_{k=0}^n \frac{n!}{k!} \dm_{n,k}(2x,-1,0,0,\dots).
\end{equation}
Some other classical orthogonal polynomials have similar descriptions, as seen in \cite[pp. 449 - 452]{chacha}. For example, the Gegenbauer polynomials may be defined with $C^{(\lambda)}_n(x) := [t^n](1-2tx+t^2)^{-\lambda}$ giving
\begin{equation}\label{geg}
  C^{(\lambda)}_n(x)  = \sum_{k=0}^n \binom{k+\lambda-1}{k}\dm_{n,k}(2x,-1,0,0,\dots)
\end{equation}
by \e{pot}.
Special cases are the  Legendre polynomials $P_n(x)$ with $\lambda=1/2$, the  Chebyshev polynomials of the second kind $U_n(x)$ with $\lambda=1$, and the usual  Chebyshev polynomials $T_n(x)$ which may found as a limit when $\lambda \to 0$. More explicitly, use $T_n(x) = [t^n](1-tx)/(1-2tx+t^2)$
to see that
\begin{equation}\label{che}
  T_n(x)  = \sum_{k=0}^n \Bigl[ \dm_{n,k}(2x,-1,0,0,\dots)-x \cdot \dm_{n-1,k}(2x,-1,0,0,\dots) \Bigr].
\end{equation}
The  Fibonacci numbers $F_n$ equal $[t^n] t/(1-t-t^2)$ and fit the same pattern with
\begin{equation}\label{fib}
  F_{n+1} = \sum_{k=0}^n  \dm_{n,k}(1,1,0,0,\dots).
\end{equation}
The identity
\begin{equation}\label{expl}
  \dm_{n,k}(x,y,0,0,\dots) = \binom{k}{n-k} x^{2k-n}y^{n-k} \qquad (n\gqs k)
\end{equation}
may also be employed in \e{her} -- \e{fib}.  For example, using \e{expl} in \e{che} yields
\begin{equation}\label{ched}
  T_n(x)  = \sum_{k=1}^n (-1)^{n-k} \frac{n}{2k} \binom{k}{n-k}(2x)^{2k-n} \qquad (n\gqs 1).
\end{equation}
In view of the usual definition of $T_n$ by $T_n(\cos \theta)=\cos n\theta$, and De Moivre's formula for the $n$th power of a complex number, it is not surprising that the polynomial
 \e{ched} already appears in the work of De Moivre, 
 predating Chebyshev's introduction of $T_n(x)$ by over 100 years. Consult \cite[p. 246]{ScIv} and \cite{Gi} for further details.

\subsection{Bernoulli numbers and polynomials}
The  Bernoulli numbers are usually defined by $B_n := n![t^n] t/(e^t-1)$. Another application of \e{recip} yields
\begin{equation}\label{ber}
  B_{n} = n!\sum_{k=0}^n  \dm_{n,k}\left(-\frac 1{2!},-\frac 1{3!},-\frac 1{4!},\dots  \right).
\end{equation}
\vskip -2mm
\noindent
Stern gave further similar expressions for $B_{n}$ in \cite{Ste}, employing the elaborate notation ${}_{\mathfrak{p}}^n\overset{k}{C'}$ for $\dm_{n,k}$.
Using the well-known power series for $\tan$ and $\arctan$ in Corollary \ref{cinv} also gives the alternative
 \begin{equation}\label{ber2}
  2^{n+2}(2^{n+2}-1) \frac{B_{n+2}}{n+2} =  n! \sum_{k=0}^n  \binom{n+k}{k} \dm_{n,k}\left(0,-\frac 1{3},0,-\frac 1{5},0,\dots  \right).
\end{equation}

Extending \e{ber} to the  Bernoulli polynomials $B_n(x) := n![t^n] te^{tx}/(e^t-1)$ results in
\begin{equation}\label{berp}
  B_{n}(x) = n!\sum_{k=0}^n  \dm_{n,k}\left(\frac{(-x)^2-(1-x)^2}{2!},\frac{(-x)^3-(1-x)^3}{3!},\dots  \right).
\end{equation}
It may be seen from \e{berp} that $B_n(x)$ is a polynomial in $x$ of degree $n$ and satisfies $B_n(1-x)=(-1)^n B_n(x)$ by \e{gsb3}.

The  N\"{o}rlund polynomials $B_{n}^{(x)} :=n![t^n] (t/(e^t-1))^x$ generalize the Bernoulli numbers in another direction. Equation \e{pot} finds
\begin{equation}\label{nor}
  B_{n}^{(x)} = n!\sum_{k=0}^n  \binom{-x}{k}\dm_{n,k}\left(\frac 1{2!},\frac 1{3!},\frac 1{4!},\dots  \right),
\end{equation}
making it more obvious by \e{bin} that $B_{n}^{(x)}$ is also a degree $n$ polynomial in $x$.
Another expression is proved in \cite[Eq. 15]{SrTo88}:
\begin{equation} \label{nor2}
  B_n^{(x)} = \sum_{k=0}^n  \binom{-x}{k} \binom{x+n}{n-k}  \binom{n+k}{k}^{-1} \stirb{n+k}{k}.
\end{equation}

\subsection{Cyclotomic polynomials}
  Ramanujan sums and cyclotomic polynomials may be expressed in terms of the M\"{o}bius function $\mu$ with
\begin{equation*}
  r_j(m)=\sum_{d \mid (m,j)}\mu(m/d) \cdot d, \qquad \Phi_n(x)=\prod_{d \mid n}\left(1 - x^d \right)^{\mu(n/d)} \quad (n\gqs 2),
\end{equation*}
respectively.
The authors in \cite[Thm. 4.2]{Mor}  write Lehmer's 1966 formula for $\Phi_n(x)$ in an appealing way using  Bell polynomials. It is even more transparent without the unnecessary factorials:
\begin{equation}\label{leh}
  \Phi_n(x)= \sum_{m=0}^\infty \left[\sum_{k=0}^m \frac{1}{k!} \dm_{m,k}\left(-\frac{r_1(n)}{1},-\frac{r_2(n)}{2},-\frac{r_3(n)}{3}, \dots \right) \right] x^m,
\end{equation}
for $n\gqs 2$, and  the easy proof is  similar to that of \e{sig}.

\section{Asymptotic expansions} \label{laps}

In this final section we discuss examples where the De Moivre polynomials play a useful role in describing the behavior of functions and sequences as a parameter goes to infinity.

\subsection{Partition asymptotics}
Hardy and Ramanujan famously gave the first asymptotic expansion for the partition function $p(n)$ in 1918. Rademacher later showed how a small alteration turned this into a rapidly converging series: \cite[Eq. (128.1)]{Ra}. A simpler, though less accurate, expansion for $p(n)$ takes the form
\begin{equation}\label{less}
  p(n)=\frac{1}{4\sqrt{3}n}e^{\pi\sqrt{2n/3}}\left(1+\frac{\hr_1}{n^{1/2}}+
\frac{\hr_2}{n^{2/2}}+ \cdots +\frac{\hr_{R-1}}{n^{(R-1)/2}}+O\left(\frac{1}{n^{R/2}}\right) \right)
\end{equation}
showing the main term along with smaller corrections. The constants $\hr_r$ are  quite complicated, but we may give them a reasonable description. First define
\begin{align}
  \alpha_j(x) & := \begin{cases} 24^{-j/2} & \text{ for $j$ even} \\
\displaystyle -24^{(1-j)/2} \frac 1x \binom{j/2}{(j-1)/2} & \text{ for $j$ odd},
\end{cases} \label{mv}\\
  \beta_\ell(x) & := \sum_{\ell/2\lqs m \lqs \ell} (-24)^{-m} \frac{x^{2m-\ell}}{(2m-\ell)!} \dm_{m,2m-\ell}\left( \binom{1/2}{1},\binom{1/2}{2},\binom{1/2}{3}, \dots \right). \label{mv2}
\end{align}

\begin{prop}
Fix $R\gqs 1$. Let $\hr_r$ be defined by $\sum_{j=0}^r \alpha_j(x)\cdot \beta_{r-j}(x)$ with $x=\pi \sqrt{2/3}$. For these values, \e{less} is true as $n \to \infty$ with an implied constant depending only on $R$.
\end{prop}
\begin{proof}
Using the first term in his expansion, Rademacher shows
\begin{equation}\label{hard}
  p(n)=\frac{1}{4\sqrt{3}n  \cdot \kappa_n} \left(1-\frac{1}{x \sqrt{n  \cdot \kappa_n}}\right) e^{x \sqrt{n \cdot \kappa_n}} \left( 1+O(e^{-x \sqrt{n}/2}) \right)
\end{equation}
for $\kappa_n :=1-1/(24n)$ in \cite[p. 278]{Ra}. Put $z:=-1/(24n)$ and write this main term as
\begin{equation*}
  \frac{e^{x \sqrt{n}}}{4\sqrt{3}n} \cdot
\frac{1}{1+z}\left(1-\frac{1}{x \sqrt{n  (1+z)}}\right) \cdot
e^{x \sqrt{n}(\sqrt{1+z} -1)}.
\end{equation*}
Treating $z$ as an independent complex variable, the middle factors are analytic for $|z|<1$ and have a Taylor expansion. Using the usual bounds on the remainder shows
\begin{equation*}
\frac{1}{1+z}\left(1-\frac{1}{x \sqrt{n  (1+z)}}\right)
=\sum_{j=0}^{J-1} a_j(\sqrt{n}) \cdot z^j +O(|z|^J) \qquad (|z|\lqs 1/2).
\end{equation*}
By the binomial theorem we find $a_j(\sqrt{n}) = (-1)^j-\binom{-3/2}{j}/(x\sqrt{n})$ and hence, with \e{mv},
\begin{equation} \label{kas}
  \frac{1}{\kappa_n} \left(1-\frac{1}{x \sqrt{n  \cdot \kappa_n}}\right) = \sum_{j=0}^{J-1} \alpha_j(x) \cdot n^{-j/2} +O( n^{-J/2}).
\end{equation}

Similarly to \e{kas},
\begin{equation*}
e^{x \sqrt{n}(\sqrt{1+z} -1)}
=\sum_{\ell=0}^{L-1} b_\ell(\sqrt{n}) \cdot z^\ell +O(|z|^L) \qquad (|z|\lqs 1/\sqrt{n}).
\end{equation*}
Then  the expansions
\begin{align*}
   e^{x \sqrt{n}(\sqrt{1+z} -1)} & =  \exp\biggl( x\sqrt{n} \sum_{j=1}^\infty \binom{1/2}{j} z^j \biggr)\\
   & =  \sum_{k=0}^\infty \frac{(x \sqrt{n})^k}{k!} \sum_{m=k}^\infty \dm_{m,k}\left( \binom{1/2}{1},\binom{1/2}{2},\binom{1/2}{3}, \dots \right) z^m
\end{align*}
imply that
\begin{equation*}
  b_\ell(\sqrt{n}) = \sum_{k=0}^\ell \frac{(x \sqrt{n})^k}{k!}  \dm_{\ell,k}\left( \binom{1/2}{1},\binom{1/2}{2},\binom{1/2}{3}, \dots \right).
\end{equation*}
After rearranging we finally obtain, with \e{mv2},
\begin{equation} \label{kas2}
  e^{x \sqrt{n}(\sqrt{\kappa_n} -1)} = \sum_{\ell=0}^{L-1} \beta_\ell(x) \cdot n^{-\ell/2} + O(n^{-L/2}).
\end{equation}
Use \e{kas} and \e{kas2} in \e{hard} to complete the proof.
\end{proof}
For example,
\begin{equation*}
  \hr_1=-\frac{72+\pi ^2}{24 \sqrt{6} \pi } \approx -0.443288, \qquad \hr_2=\frac{432+\pi ^2}{6912} \approx  0.0639279.
\end{equation*}
We see from \e{mv}, \e{mv2} that $C_r$ is $1/\pi$ times a polynomial in $\pi$ of degree $r+1$,  and so we cannot expect much simplification. Of course \e{hard} is more accurate than \e{less}, and including more terms of Rademacher's series is better still. The numbers $p(n)$ are examples of Fourier coefficients of weakly holomorphic modular forms and Rademacher type series exist for all of these.  In the next examples though, and many other situations, expansions of the form \e{less} are the best available.

\subsection{Laplace's method} \label{lap}
A  simple example of  Laplace's method, (and where it overlaps with the saddle-point method), is the following. Suppose that $f(z)$ and $g(z)$ are holomorphic functions on a domain containing the interval $[-1,1]$ and  real-valued on this interval. The integral
\begin{equation} \label{i(n)}
   I(n):= \int_{-1}^{1} e^{n \cdot f(z)} g(z) \, dz
\end{equation}
can be accurately  estimated for large $n \in \R$ if $f'(z)$ has a simple zero at $z=0$ (called the saddle-point) and $f(z)<f(0)$ for all non-zero $z\in [-1,1]$.
In that case, the integral becomes highly concentrated at $0$ and as $n \to \infty$,
\begin{equation}\label{perr}
 I(n) = e^{n \cdot f(0)} \left(\sum_{r=0}^{R-1}   \frac{\G(r+1/2) \cdot \Psi_{2r}(f,g)}{n^{r+1/2}} + O\left(\frac{1}{n^{R+1/2}} \right)  \right)
\end{equation}
for certain constants $\Psi_{2r}(f,g)$. If we write the expansions of $f$ and $g$ at $0$ as
\begin{equation*}
  f(z)-f(0)=-\sum_{m=0}^\infty a_m z^{m+2}, \qquad g(z) = \sum_{m=0}^\infty b_m z^{m}
\end{equation*}
then we have
\begin{equation}\label{perr2}
  \Psi_{s}(f,g) = a_0^{-(s+1)/2} \sum_{m=0}^s b_{s-m} \sum_{k=0}^m \binom{-(s+1)/2}{k} \dm_{m,k}\left( \frac{a_1}{a_0},\frac{a_2}{a_0}, \dots \right).
\end{equation}
Laplace found the main term of \e{perr} with $r=0$ and $\Psi_{0}(f,g)=a_0^{-1/2} b_0$ in the 18th century. Perron in \cite{Pe17} proved \e{perr} rigorously, giving a formula for the constants $\Psi_{2r}(f,g)$ in terms of a derivative.   It is then easy to obtain \e{perr2} with \e{pot}. See \cite{Nem13} for a discussion of this, and the weaker conditions on $f$ and $g$ that are possible in Laplace's method. The general forms of \e{perr} and \e{perr2} in Perron's saddle-point method are also described in Corollary 5.1 and Proposition 7.2  of \cite{OSper}.


For an important application to the gamma function write
\begin{equation*}
  \G(n+1)= \int_0^\infty e^{-t} t^n \, dt = n^{n+1}\int_{-1}^\infty e^{n(\log(1+z)-1-z)}\, dz.
\end{equation*}

\begin{cor}[Stirling's approximation] \label{stirgam}
As real $n \to \infty$
\begin{equation} \label{gmx}
  \G(n+1)= \sqrt{2\pi n}\left(\frac{n}{e}\right)^n \left(1+\frac{\g_1}{n}+\frac{\g_2}{n^2}+ \cdots + \frac{\g_{k}}{n^{k}} +O\left(\frac{1}{n^{k+1}}\right)\right).
\end{equation}
\end{cor}

From \e{perr2} after simplifying, we obtain the coefficients
\begin{equation} \label{stx}
  \g_m =  \sum_{j=0}^{2m}  \frac{(2m+2j-1)!!}{(-1)^j j!} \dm_{2m,j}\left(\frac{1}{3},\frac{1}{4},\frac{1}{5}, \cdots \right),
\end{equation}
as in  \cite[Sect. 5]{Pe17}, \cite[Sect. 8.1]{OSper}, where $(2k-1)!!$ means the product of the odd numbers less than $2k$ and equals $(2k)!/(2^k k!)$ for $k \gqs 0$.
Stirling's series for $\log \G(n)$, see  \cite[Eq. (1.4.5)]{AAR}, is
\begin{equation} \label{stid}
\log \G(n+1) = \log\left( \sqrt{2\pi n}\left(\frac{n}{e}\right)^n\right) +
\frac{B_2}{2\cdot 1\cdot  n}+\frac{B_4}{4\cdot 3\cdot n^3}+ \cdots + \frac{B_{2k}}{2k(2k-1)n^{2k-1}}+\bo{\frac{1}{n^{2k+1}}},
\end{equation}
and may be exponentiated to produce another expression for $\g_m$:
\begin{equation} \label{stx2}
  \g_m =  \sum_{j=0}^{m}  \frac{1}{j!} \dm_{m,j}\left(\frac{B_2}{2\cdot 1},0,\frac{B_4}{4\cdot 3},0, \cdots \right).
\end{equation}
Replacing the Bernoulli numbers in \e{stx2} with Riemann zeta  values, \cite[Eq. (1.3.4)]{AAR}, gives the equivalent equality
\begin{equation} \label{stx3}
  \g_m =  \sum_{j=0}^{m}  \frac{1}{j!} \dm_{m,j}\left(\frac{\zeta(-1)}{-1},\frac{\zeta(-2)}{-2},\frac{\zeta(-3)}{-3}, \cdots \right),
\end{equation}
and this has a pleasing symmetry with
\begin{equation}\label{zet}
  \G(z+1) = \sum_{m=0}^\infty \left[\sum_{j=0}^m \frac{1}{j!} \dm_{m,j}\left(\g,\frac{\zeta(2)}{2},\frac{\zeta(3)}{3}, \cdots \right) \right] (-z)^m,
\end{equation}
where $\g$ is Euler's constant. We have  convergence in \e{zet} for $|z|<1$ and it is proved by taking the log of the Weierstrass product for $1/\G(z)$, expanding this into a series, and then exponentiating. Further formulas for $\g_m$ appear in \cite[p. 267]{Comtet} and \cite[Sect. 3]{Nem13}.
See \cite{Gel}, \cite[Sect. 5.3]{ScIv} for the history of Stirling's series. The common form \e{stid} we use today, involving Bernoulli numbers, is due to De Moivre who simplified Stirling's treatment.

The reciprocal of the gamma function has a similar asymptotic expansion:
\begin{equation} \label{gmx2}
  \frac 1{\G(n+1)}= \frac 1{\sqrt{2\pi n}}\left(\frac{e}{n}\right)^n \left(1-\frac{\g_1}{n}+\frac{\g_2}{n^2}- \cdots + (-1)^{k}\frac{\g_{k}}{n^{k}} +O\left(\frac{1}{n^{k+1}}\right)\right),
\end{equation}
and the fact that the same coefficients appear is a consequence of \e{stid} containing only odd powers.
Of course, since $\G(n+1) = n!$ for $n\in \Z_{\gqs 0}$, the results \e{gmx} and \e{gmx2} may be used to estimate binomial coefficients or other factorial expressions as their entries go to infinity.
See \cite{Boyd} for a generalization of Corollary \ref{stirgam} with $n$ replaced by $z \in \C$ and $|z| \to \infty$ with $|\arg z|<\pi/2$.

\subsection{A further example}
The integral
\begin{equation} \label{simpl}
  I_\alpha(n):=\int_1^\infty  (\log z)^n e^{-\alpha z}   \, dz \qquad(\alpha>0),
\end{equation}
requires more involved methods than \e{i(n)} to find its behavior as $n \to \infty$.
Since $\log \log z$ does not have a local maximum, we look for a saddle-point $z_0$ for the whole integrand. To locate it, recall the Lambert $W$ function.  For $x\gqs 0$ it may be defined  as the inverse to $x \mapsto xe^x$  so that
$W(xe^x)=x$. Then $W(x)$ is non-negative and increasing
for $x\gqs 0$, satisfying $W(x)\lqs \log x$  when $x\gqs e$.
An easy calculation finds
$$
z_0=e^{W(n/\alpha)}= \frac{n/\alpha}{W(n/\alpha)}.
$$
In this case we may still use Laplace's method according to the general procedures in \cite[Sect. B6]{Fl09}. Part of the series \e{hot} has to be exponentiated and the polynomials $\dm_{n,k}$  keep track of all the components.
Recalling  \e{lnu},  define the rational functions
\begin{equation*}
  a_r(v)  :=\sum_{j=0}^{2r} \frac{(2r+2j-1)!!}{j!} \left( \frac{v^2}{v+1}\right)^{j+r}
   \dm_{2r,j}(\ell_3(v), \ell_4(v), \dots).
\end{equation*}
The next result is proved in \cite[Sect. 4]{OSxi}.

\begin{theorem}  \label{ian}
Suppose $\alpha>0$ and set $u :=W(n/\alpha)$. Then as $n \to \infty$ we have
\begin{equation} \label{maini}
 I_{\alpha}(n) =  \frac{\sqrt{2\pi} u}{\sqrt{(1+u)n}} \left(\frac{u}{e^{1/u}}\right)^n \left( 1+  \sum_{r=1}^{R-1}\frac{a_r(u)}{n^r}+  O\left( \frac{(\log n)^R}{n^R}\right) \right)
\end{equation}
where the implied constant  depends only on $R \gqs 1$  and $\alpha$. Also $a_r(u)  \ll (\log n)^r$.
\end{theorem}

 A more general result allows suitable functions $f(z)$ to be included in the integrand in \e{simpl} and this is used to prove the asymptotic expansion of the $n$th Taylor coefficient of a normalized version of the Riemann zeta function $\zeta(s)$ at the central symmetric point $s=1/2$ as $n\to \infty$. See \cite[Thm. 1.5]{OSxi}, giving an explicit version of  \cite[Thm. 9]{GORZ} by expressing the expansion coefficients   in terms of De Moivre polynomials.

{\small \bibliography{bell-bib} }

{\small 
\vskip 5mm
\noindent
\textsc{Dept. of Math, The CUNY Graduate Center, 365 Fifth Avenue, New York, NY 10016-4309, U.S.A.}

\noindent
{\em E-mail address:} \texttt{cosullivan@gc.cuny.edu}
}

\end{document}